\documentclass[11pt,a4paper]{article}

\usepackage{authblk}

\usepackage{mathrsfs}
\usepackage{amsfonts}
\usepackage{cite}
\usepackage{amsmath,amssymb,amscd,amsthm}
\usepackage{graphics}
\usepackage{enumerate}
\input xy
\xyoption{all}

\numberwithin{equation}{section}

\usepackage{scalefnt}   

\usepackage{comment}
\usepackage{setspace}

\usepackage[paperwidth=210mm,
            paperheight=297mm,
            a4paper,
            twoside,
            top=0.8in, left=0.9in, right=0.9in, bottom=1in]{geometry}

\usepackage[active]{srcltx}

\theoremstyle{definition}
\newtheorem{lemma}{Lemma}[section]
\newtheorem*{claim}{Claim}
\newtheorem{prop}[lemma]{Proposition}
\newtheorem{cor}[lemma]{Corollary}

\newtheorem{thm}[lemma]{Theorem}
\newtheorem{ex}[lemma]{Example}

\newtheorem{case}{Case}

\newcommand{\C}{\mathbb{C}}
\newcommand{\N}{\mathbb{N}}
\newcommand{\Z}{\mathbb{Z}}

\def\a{\alpha}

\def\im{\mathrm{Im}(\varphi_{\alpha\beta})}

\usepackage{float}
\restylefloat{figure}
\allowdisplaybreaks

\usepackage[colorlinks=true, linkcolor=magenta, urlcolor=blue, citecolor=blue]{hyperref}

\begin{document}
\title{Simple $\mathcal{W}\ltimes\widehat{H_4}$-modules from tensor products}

\author{Dashu Xu\footnote{E-mail: dox@mail.ustc.edu.cn}}

\date{}
\maketitle
\begin{abstract}
This paper investigates simple modules of the semi-direct product algebra $\mathcal{W}\ltimes\widehat{H_4}$, where $\mathcal{W}$ is the Witt algebra and $\widehat{H_4}$ is the loop Diamond algebra.
We first use simple modules over the Weyl algebra to construct a family of simple $\mathcal{W}\ltimes\widehat{H_4}$-modules.
Then, we classify simple $\mathcal{W}\ltimes\widehat{H_4}$-modules that are free $U(\C L_0\oplus\C a_0)$-modules of rank 1.
Finally, we give a necessary and sufficient condition for finitely many simple $U(\C L_0\oplus\C a_0)$-free modules to be simple, and then determine their isomorphism classes.

\end{abstract}

\section{Introduction and Preliminaries}
Throughout this paper, we use $\N$, $\Z$, $\C$ and $\C^*$ to denote the sets of non-negative integers, integers, complex numbers and non-zero complex numbers respectively.
For an arbitrary Lie algebra $\mathfrak{g}$, we employ $U(\mathfrak{g})$ to denote its universal enveloping algebra.

Let $\C[t^{\pm1}]$ be the algebra of Laurent polynomials.
The Witt algebra is defined to be the derivation algebra of $\C[t^{\pm1}]$.
It has a $\C$-linear basis $\left\{t^{n+1}\partial_t\mid n\in\Z\right\}$ with the Lie bracket $$[t^{m+1}\partial_t,t^{n+1}\partial_t]=(n-m)t^{m+n+1}\partial_t.$$ Denote the Witt algebra by $\mathcal{W}$.

For an arbitrary Lie algebra $\mathfrak{g}$, we have the corresponding loop algebra, which is denoted by $\widehat{\mathfrak{g}}$. 
The underlying vector space of $\widehat{\mathfrak{g}}$ is $\mathfrak{g}\otimes\C[t^{\pm1}]$. 
For $x,y\in\mathfrak{g}$ and $m,n\in\Z$, the Lie bracket is defined as $$[x\otimes t^m, y\otimes t^n]=[x,y]\otimes t^{m+n}.$$
The Witt algebra acts on the loop algebra as derivations.
More precisely, there exists a Lie algebra homomorphism
\begin{align*}
	\varphi:\mathcal{W}&\to\mathrm{Der}(\widehat{\mathfrak{g}})\\
	t^{m+1}\partial_t&\mapsto \varphi(t^{m+1}\partial_t):\; x\otimes t^n\mapsto n(x\otimes t^{m+n}),
\end{align*}
where $x\in\mathfrak{g}$ and $m,n\in\Z$.
Hence, we have the semi-direct product $\mathcal{W}\ltimes\widehat{\mathfrak{g}}$.
In particular, if $\mathfrak{g}$ is a complex simple Lie algebra, there is a 2-dimensional central extension of $\mathcal{W}\ltimes\widehat{\mathfrak{g}}$, which is the {conformal current algebra} studied in \cite{K}.
Weight modules of $\mathcal{W}\ltimes\widehat{\mathfrak{g}}$ have been well understood due to the work of \cite{K,LPX,CLW}.
In the case of $\mathfrak{g}=\mathfrak{sl}_2$, some non-weight simple modules are constructed in \cite{CH,CY}.

This paper is a continuation of \cite{CX}.
We study representations of the algebra $\mathcal{W}\ltimes\widehat{H_4}$, where $\mathcal{W}$ is the Witt algebra and $H_4$ is the Diamond algebra.
The Diamond algebra is a 4-dimensional Lie algebra that is powerful in physics. 
A Wess-Zumino-Witten model is built upon this algebra\cite{NW}.
It is also relevant to Leibniz algebras\cite{UKO} and special functions\cite{M}.
The affinization of the Diamond algebra has been well-explored in \cite{BJP}.
Research has also been conducted on the vertex operator algebra associated with the Diamond algebra \cite{B}.

Specifically, the Diamond algebra $H_{4}$ has a $\C$-basis $\left\{a,b,c,d\right\}$ with the non-vanishing Lie brackets:
\[
[a,b]=c,\quad\quad [d,a]=a,\quad\quad [d,b]=-b.
\]
For $n\in\Z$ and $x\in H_4$, denote by $L_n=t^{n+1}\partial_t\in\mathcal{W}$ and $x_n=x\otimes t^n\in\widehat{H_4}$.
Then, the semi-direct product algebra $\mathcal{W}\ltimes\widehat{H_4}$ has a $\C$-linear basis $\left\{L_n,a_n,b_n,c_n,d_n\mid n\in\Z\right\}$, and the non-vanishing Lie brackets are as follows:
\begin{align*}
	[L_m,L_n]&=(n-m)L_{m+n},\quad\quad [L_m,x_n]=nx_{m+n},\quad\quad [a_m,b_n]=c_{m+n},\\
	[d_m,a_n]&=a_{m+n},\quad\quad\quad\quad\quad\quad [d_m,b_n]=-b_{m+n},
\end{align*}
where $x\in\left\{a,b,c,d\right\}$.
\textbf{In the rest of this paper}, we denote the algebra $\mathcal{W}\ltimes\widehat{H_4}$ by $L$.
Meanwhile, denote by $B=\C L_0\oplus\C a_0$ and $H=\C L_0\oplus\C d_0$, respectively.

In Section \ref{22}, we use simple modules of the Weyl algebra to construct a large family of simple $\mathcal{W}\ltimes\widehat{H_4}$-modules.
This construction is motivated by the Shen-Larsson functor\cite{S,L}.
Then, inspired by the work of \cite{N}, in Section \ref{33}, we classify simple $\mathcal{W}\ltimes\widehat{H_4}$-modules that are free $U(B)$-modules of rank 1.
Finally, in Section \ref{44}, we study the tensor products of finitely many $U(B)$-free simple modules. 
This work is built upon previous work in \cite{CY}.
A necessary and sufficient condition for such a tensor product to be simple is obtained and the isomorphism classes are determined.

\section{The simple module $\mathcal{F}_{\alpha\beta}(P,V)$}
\label{22}
Denote by $$\mathcal{R}_2=\C[x_0^{\pm1},x_1^{\pm1},\partial_{x_0},\partial_{x_1}]$$
the Weyl algebra of degree 2.
Then we have $\mathcal{R}_2=\mathcal{R}_0\otimes\mathcal{R}_1$, where for $i=0,1$, $\mathcal{R}_i=\C[x_i^{\pm1},\partial_{x_i}]$ is the Weyl algebra of degree 1.
For $i=0,1$, we simply use $\partial_i$ to denote $x_i\partial_{x_i}$.
Suppose $\mathfrak{b}=\C h\oplus\C e$ is a 2-dimensional Lie algebra with the Lie bracket $[h,e]=e$.

For $\alpha\in\C$ and $\beta\in\C^*$, define a linear map as follows:
\begin{align*}
	\varphi_{\alpha\beta}:U(L)&\to \mathcal{R}_2\otimes U(\mathfrak{b})\\
	L_n&\mapsto x_0^n(\partial_0+n\alpha)\otimes 1,\\
	d_n&\mapsto x_0^n\partial_1 \otimes 1 + nx_0^n\otimes e,\\
	a_n&\mapsto \beta x_0^nx_1\partial_1 \otimes 1 + x_0^nx_1\otimes (h-ne),\\
	b_n&\mapsto x_0^nx_1^{-1}\otimes 1,\\
	c_n&\mapsto -\beta x_0^n\otimes 1.
\end{align*}
Then the following statements are valid.
\begin{prop}
\label{alghomo}
The map $\varphi_{\alpha\beta}$ is a homomorphism of associative algebras.
\end{prop}
\begin{proof}
For $i=0,1$, we have $[\partial_i,x_i]=x_i$ and $[\partial_i,x_i^{-1}]=-x_i^{-1}$.
Hence, it is straightforward to see that $[x_0^n\partial_0, x_0^m]=mx_0^{n+m}$ for all $n,m\in\Z$.
It follows that 
\[
[\varphi_{\alpha\beta}(L_n),\varphi_{\alpha\beta}(x_m)]=m\varphi_{\alpha\beta}(x_{n+m}),
\]
where $x\in\left\{a,b,c,d\right\}$.
We also have the following computations:
\begin{align*}
 &\;[\varphi_{\alpha\beta}(d_n),\varphi_{\alpha\beta}(a_m)]\\
=&\; [x_0^n\partial_1 \otimes 1 + nx_0^n\otimes e, \beta x_0^mx_1\partial_1 \otimes 1 + x_0^mx_1\otimes (h-me)]\\
=&\;\beta[x_0^n\partial_1 \otimes 1, x_0^mx_1\partial_1 \otimes 1]+[x_0^n\partial_1 \otimes 1,x_0^mx_1\otimes (h-me)]+[nx_0^n\otimes e,x_0^mx_1\otimes (h-me)]\\
=&\;\beta x_0^{n+m}x_1\partial_1\otimes 1+x_0^{n+m}x_1\otimes(h-me)+nx_0^{n+m}x_1\otimes[e,h-me]\\
=&\;\beta x_0^{n+m}x_1\partial_1\otimes 1+x_0^{n+m}x_1\otimes(h-me-ne)\\
=&\;\varphi_{\alpha\beta}(a_{n+m}),\\[12pt]
&\;[\varphi_{\alpha\beta}(d_n),\varphi_{\alpha\beta}(b_m)]\\
=&\;[x_0^n\partial_1 \otimes 1 + nx_0^n\otimes e, x_0^mx_1^{-1}\otimes1]\\
=&\;-x_0^{n+m}x_1^{-1}\otimes1\\
=&\;-\varphi_{\alpha\beta}(b_{n+m}),\\[12pt]
&\;[\varphi_{\alpha\beta}(a_n),\varphi_{\alpha\beta}(b_m)]\\
=&\;[\beta x_0^nx_1\partial_1 \otimes 1 + x_0^nx_1\otimes (h-ne),x_0^mx_1^{-1}\otimes1]\\
=&\;\beta x_0^{n+m}[x_1\partial_1, x_1^{-1}]\otimes1\\
=&\;-\beta x_0^{n+m}\otimes1\\
=&\;\varphi_{\alpha\beta}(c_{n+m}).
\end{align*}
Hence $\varphi_{\alpha\beta}$ is a homomorphism of associative algebras.
\end{proof}

For a simple $\mathcal{R}_2$-module $P$ and a simple $U(\mathfrak{b})$-module $V$, the tensor product $P\otimes V$ is a module over the tensor product algebra $\mathcal{R}_2\otimes U(\mathfrak{b})$, hence can be lifted to a $U(L)$-module through the algebra homomorphism $\varphi_{\alpha\beta}$.
We denote such a $U(L)$-module by $\mathcal{F}_{\alpha\beta}(P,V)$.
More precisely, $\mathcal{F}_{\alpha\beta}(P,V)$ is a $U(L)$-module under the actions below:
\begin{equation}\label{mod}
\begin{split}
L_n(p\otimes v)&=x_0^n(\partial_0+n\alpha)p\otimes v,\\
d_n(p\otimes v)&=x_0^n\partial_1 p\otimes v + nx_0^np\otimes ev,\\
a_n(p\otimes v)&=\beta x_0^nx_1\partial_1 p\otimes v + x_0^nx_1p\otimes (h-ne)v,\\
b_n(p\otimes v)&=x_0^nx_1^{-1}p\otimes v,\\
c_n(p\otimes v)&=-\beta x_0^np\otimes v,
\end{split}
\end{equation}
where $p\in P$ and $v\in V$.

In the following, we shall study whether the $U(L)$-module $\mathcal{F}_{\alpha\beta}(P,V)$ is simple. Before this, we recall an elementary lemma first.
\begin{lemma}[Theorem 3.2 in \cite{D}]
\label{density}
Let $P$ be a simple module over a unital associative algebra $\mathcal{A}$ that has a countable basis.
Then for linearly independent elements $p_1,\ldots,p_m\in P$ and any elements $u_1,\ldots,u_m\in P$, there exists
some $a\in \mathcal{A}$ such that $ap_i = u_i$ for $i = 1,\ldots,m$.
\end{lemma}
In order to study the module $\mathcal{F}_{\alpha\beta}(P,V)$, it is necessary to characterize the images of the map $\varphi_{\alpha\beta}$.
Denote by $\im$ the image of $\varphi_{\alpha\beta}$.
\begin{prop}
\label{image}
The elements
\[
x_0^{\pm1}\otimes1,\quad \partial_{x_0}\otimes1,\quad x_1\otimes e,\quad x_1^{-1}\otimes1,\quad \partial_{x_1}\otimes1,\quad 1\otimes e,\quad\mbox{and}\quad 1\otimes h
\]
lie in $\im$.
\end{prop}
\begin{proof}
The proof is straightforward from the below equations:
\begin{align*}
x_0^{\pm1}\otimes1&=-\beta^{-1}\varphi_{\alpha\beta}(c_{\pm1}),\\
\partial_{x_0}\otimes1&=-\beta^{-1}\varphi_{\alpha\beta}(c_{-1}L_0),\\
x_1\otimes e&=\beta^{-1}\varphi_{\alpha\beta}(c_{-1}a_1-c_0a_0),\\
x_1^{-1}\otimes1&=\varphi_{\alpha\beta}(b_0),\\
1\otimes e&=\beta^{-1}\varphi_{\alpha\beta}\left((c_{-1}a_1-c_0a_0)b_0\right),\\
1\otimes h&=\varphi_{\alpha\beta}(b_0a_0-\beta d_0),\\
\partial_{x_1}\otimes1&=\varphi_{\alpha\beta}(b_0d_0).\qedhere
\end{align*}
\end{proof}
Now, we are in a position to prove the \textbf{main result} of this section.
\begin{thm}
Assume that $V$ is infinite-dimensional. Then the $U(L)$-module $\mathcal{F}_{\alpha\beta}(P,V)$ is simple. 
\end{thm}
\begin{proof}
Since $V$ is infinite-dimensional, we know that $e\in\mathfrak{b}$ acts on $V$ injectively.
For a non-zero submodule $N$, suppose $m=\sum_{i=1}^{m}p_i\otimes v_i$ is any non-zero element of $N$ with $p_1,\ldots,p_m$ being linearly independent.
Take any $0\ne p\in P$, set $u_1=p$ and $u_2=\cdots=u_m=0$.
Then by Lemma \ref{density}, there exists some $a\in\mathcal{R}_2$ such that $ap_i=u_i$.
From Proposition \ref{image}, we deduce that there exists some $n\gg 0$ such that $a\otimes e^n\in\im$.
It follows that $0\ne p\otimes u\in N$, where $u=e^nv_1$. 
Note that $V=U(\mathfrak{b})u$, by Proposition \ref{image} again, we know that $p\otimes V\subset N$. Since $p\in P$ is arbitrary, we conclude that $N=\mathcal{F}_{\alpha\beta}(P,V)$.
\end{proof}
Taking $P$ and $V$ to be some specific modules, we obtain some interesting simple $U(L)$-modules.
Two examples are listed below.

For $\mathbf{a}=(a_0,a_1)\in\C^2$, the space $M_{\mathbf{a}}=\bigoplus_{n_0,n_1\in\Z}\C x_0^{n_0}x_1^{n_1}$ is a simple $\mathcal{R}_2$-module, where $x_i^{\pm1}$ acts as multiplication and $\partial_i\left(x_0^{n_0}x_1^{n_1}\right)=(a_i+n_i)x_0^{n_0}x_1^{n_1}$.
Recall that a $\mathcal{W}\ltimes\widehat{H_4}$-module is called a weight module if the subalgebra $H=\C L_0\oplus\C d_0$ acts semisimply.
\begin{ex}
	\label{wei}
Suppose $P=M_{\mathbf{a}}$ and $V$ is an infinite-dimensional simple $U(\mathfrak{b})$-module.
Then
\[
\mathcal{F}_{\alpha\beta}(M_{\mathbf{a}},V)=\bigoplus_{n_0,n_1\in\Z}x_0^{n_0}x_1^{n_1}\otimes V
\]
is a simple weight $U(L)$-module with infinite-dimensional weight spaces.
\end{ex}
For $\boldsymbol{\lambda}=(\lambda_0,\lambda_1)\in(\C^*)^2$, let $\Omega(\boldsymbol{\lambda})=\C[\partial_0,\partial_1]$.
Then $\Omega(\boldsymbol{\lambda})$ is a simple $\mathcal{R}_2$-module, where $\partial_i$ acts as multiplication and $x_i^nf(\partial_0,\partial_1)=\lambda_i^nf(\partial_0-\delta_{i0}n,\partial_1-\delta_{i1}n)$.
\begin{ex}
	\label{iuh}
If $P$ is the simple $\mathcal{R}_2$-module $\Omega(\boldsymbol{\lambda})$, then
\[
\mathcal{F}_{\alpha\beta}(\Omega(\boldsymbol{\lambda}),V)=\bigoplus_{n_0,n_1\in\Z}\partial_0^{n_0}\partial_1^{n_1}\otimes V
\]
is a simple $U(L)$-module and is free of infinite rank when restricted to $U(H)$.
\end{ex}
Next, we determine the isomorphism classes of the simple modules $\mathcal{F}_{\alpha\beta}(P,V)$.
\begin{thm}
Assume that $\alpha,\xi\in\C$ and $\beta,\mu\in\C^*$.
Suppose $P,Q$ are simple $\mathcal{R}_2$-modules and $V$ is an infinite-dimensional simple $U(\mathfrak{b})$-module. 
Then we have $\mathcal{F}_{\alpha\beta}(P,V)\cong\mathcal{F}_{\xi\mu}(Q,W)$ if and only if $(\alpha,\beta)=(\xi,\mu)$, $P\cong Q$ as $\mathcal{R}_2$-modules, and $V\cong W$ as $U(\mathfrak{b})$-modules.
\end{thm}
\begin{proof}
Suppose $\phi:\mathcal{F}_{\alpha\beta}(P,V)\to\mathcal{F}_{\xi\mu}(Q,W)$ is an isomorphism map.
The action of $c_0$ implies $\beta=\mu$.
Fix $0\ne p\in P$ and $0\ne v\in V$, assume that $\phi(p\otimes v)=\sum_{i=1}^{m}q_i\otimes w_i$, where $q_1,\ldots,q_m$ are linearly independent.
For any $q\in Q$, by Lemma \ref{density}, there exists $a_q\in\mathcal{R}_2$ such that $a_qq_1=q$ and $a_qq_i=0$ for $i=2,\ldots,m$.
Using Proposition \ref{image} and noting that $\beta=\mu$, we know that there exists $n\gg0$ such that $a_q\otimes e^n$ has a same preimage of the maps $\varphi_{\alpha\beta}$ and $\varphi_{\xi\mu}$, respectively. 
Hence, we have $\phi(a_qp\otimes e^nv)=q\otimes e^nw_1$, where $e^nv$ and $e^nw_1$ are both non-zero.
By Proposition \ref{image} again, there exists a $U(\mathfrak{b})$-module isomorphism $\phi_2:W\to V$ such that $\phi^{-1}(q\otimes w)=a_qp\otimes\phi_2(w)$. 
It follows that the map $\phi_1:Q\to P, q\mapsto a_qp$ is an isomorphism of $\mathcal{R}_2$-modules. 
Now, it is easy to see that $\alpha=\xi$.
\end{proof}
If $V$ is finite-dimensional, then $\mathrm{dim}V=1$, $eV=0$ and $h$ acts on $V$ as a scalar $\epsilon\in\C$.
In this case, we denote $V$ by $\C_\epsilon$.
Define 
\begin{equation}
\label{operator}
Q=b_0a_0+c_0d_0\in U(L).
\end{equation}
\begin{prop}
The modules $\mathcal{F}_{\xi\mu}(Q,\C_\epsilon)$ and $\mathcal{F}_{\alpha\beta}(P,V)$ with $V$ being infinite-dimensional are always non-isomorphic.
\end{prop}
\begin{proof}
By direct computations, we know that $Q$ acts on $\mathcal{F}_{\alpha\beta}(P,\C_\epsilon)$ as a scalar $\epsilon\in\C$.
For $p\in P$ and $v\in V$, $Q(p\otimes v)=p\otimes hv$.
Hence, if $\mathcal{F}_{\xi\mu}(Q,\C_\epsilon)\cong \mathcal{F}_{\alpha\beta}(P,V)$, then $h$ acts on $V$ as a scalar $\epsilon\in\C$, which is impossible since $V$ is an infinite-dimensional simple $U(\mathfrak{b})$-module.
\end{proof}
The structure of $\mathcal{F}_{\alpha\beta}(P,\C_\epsilon)$ is complicated. We give two examples to illustrate this.
\begin{ex}
$P=M_{\mathbf{a}}$ is the simple $\mathcal{R}_2$-module defined before Example \ref{wei}.
Then we have
\[
\mathcal{F}_{\alpha\beta}(M_{\mathbf{a}},\C_\epsilon)=\bigoplus_{n_0,n_1\in\Z}\C x_0^{n_0}x_1^{n_1},
\]
and the $U(L)$-actions are as follows:
\begin{equation*}
	\begin{split}
		L_n(x_0^{n_0}x_1^{n_1})&=(a_0+n_0+n\alpha)x_0^{n+n_0}x_1^{n_1},\\
		d_n(x_0^{n_0}x_1^{n_1})&=(a_1+n_1)x_0^{n+n_0}x_1^{n_1},\\
		a_n(x_0^{n_0}x_1^{n_1})&=(\beta a_1+\beta n_1+\epsilon)x_0^{n+n_0}x_1^{n_1+1},\\
		b_n(x_0^{n_0}x_1^{n_1})&=x_0^{n+n_0}x_1^{n_1-1},\\
		c_n(x_0^{n_0}x_1^{n_1})&=-\beta x_0^{n+n_0}x_1^{n_1}.
	\end{split}
\end{equation*}
It is easy to see that $\mathcal{F}_{\alpha\beta}(M_{\mathbf{a}},\C_\epsilon)$ is simple if and only if $\beta a_1+\beta n+\epsilon\ne0$ for any $n\in\Z$.
\end{ex}
\begin{ex}
Take $P=\Omega(\boldsymbol{\lambda})$ to be the simple $\mathcal{R}_2$-module defined before Example \ref{iuh}. Then we have
\[
\mathcal{F}_{\alpha\beta}(\Omega(\boldsymbol{\lambda}),\C_\epsilon)=\C[\partial_0,\partial_1],
\]
and the actions of $U(L)$ are as follows:
\begin{equation}
	\begin{split}
		L_nf(\partial_0,\partial_1)&=\lambda_1^n(\partial_0+n\alpha)f(\partial_0-n,\partial_1),\\
		d_nf(\partial_0,\partial_1)&=\lambda_1^n\partial_1f(\partial_0-n,\partial_1),\\
		a_nf(\partial_0,\partial_1)&=\lambda_1^n\lambda_2(\beta \partial_1+\epsilon)f(\partial_0-n,\partial_1-1),\\
		b_nf(\partial_0,\partial_1)&=\lambda_1^n\lambda_2^{-1}f(\partial_0-n,\partial_1+1),\\
		c_nf(\partial_0,\partial_1)&=-\beta \lambda_1^nf(\partial_0,\partial_1),
	\end{split}
\end{equation}
where $f(\partial_0,\partial_1)\in\C[\partial_0,\partial_1]$.
The $U(L)$-module $\mathcal{F}_{\alpha\beta}(\Omega(\boldsymbol{\lambda}),\C_\epsilon)$ is always simple, as proved in Proposition 3.3 of \cite{CX}.
\end{ex}
We have some partial results.
For $a\in\C$, the space $M_a=\bigoplus_{n_1\in\Z}\C x_1^{n_1}$ is an $\mathcal{R}_1$-module, where $x_1^{\pm1}$ acts as multiplication and $\partial_{1}x_1^{n_1}=(a+n_1)x_1^{n_1}$.
The following results can be found in Theorem 2.3 of \cite{GHL}.
\begin{thm}
Any simple $\partial_1$-torsion $\mathcal{R}_2$-module is isomorphic to one of $P_0\otimes M_a$, where $P_0$ is a simple $\mathcal{R}_0$-module and $a\in\C$.
\end{thm}
In the case $P=P_0\otimes M_a$, the actions of $U(L)$ on $\mathcal{F}_{\alpha\beta}(P_0\otimes M_a,\C_\epsilon)$ are as follows:
\begin{equation*}
	\begin{split}
		L_n(p_0\otimes x_1^{n_1})&=x_0^n\partial_0p_0\otimes x_1^{n_1}+n\alpha x_0^np_0\otimes x_1^{n_1},\\
		d_n(p_0\otimes x_1^{n_1})&=x_0^np_0\otimes(a+n_1)x_1^{n_1},\\
		a_n(p_0\otimes x_1^{n_1})&=x_0^np_0\otimes(\beta a+\beta n_1+\epsilon)x_1^{n_1+1},\\
		b_n(p_0\otimes x_1^{n_1})&=x_0^np_0\otimes x_1^{n_1-1},\\
		c_n(p_0\otimes x_1^{n_1})&=-\beta x_0^np_0\otimes x_1^{n_1},
	\end{split}
\end{equation*}
where $p_0\otimes x_1^{n_1}\in P_0\otimes M_a$.
The following statements are straightforward.
\begin{prop}
The module $\mathcal{F}_{\alpha\beta}(P_0\otimes M_a,\C_\epsilon)$ is simple if and only if $\beta a+\beta n+\epsilon\ne0$ for any $n\in\Z$.
\end{prop}

\section{Simple $U(B)$-free modules of rank 1}
\label{33}
In this section, we will classify all the simple $U(L)$-modules that are free $U(B)$-modules of rank 1.

First, we present the following observations.
Let $\mathcal{D}=\C[t,\partial_t]$ be the algebra of differential operators. 
For $\alpha,\gamma\in\C$, $\beta\in\C^*$ and $g(t)\in\C[t]$, define a linear map as follows:
\begin{align*}
		\phi_{\alpha\beta\gamma g}:U(L)&\to\mathcal{R}_0\otimes\mathcal{D}\\
		L_n&\mapsto x_0^n(\partial_0+n\alpha)\otimes1,\\
		d_n&\mapsto x_0^n\otimes\left(\beta^{-1}tg(t)+\beta^{-1}\gamma+t \partial_{t}\right),\\
		a_n&\mapsto x_0^n\otimes t ,\\
		b_n&\mapsto x_0^n\otimes\left(g(t)+\beta\partial_t\right),\\
		c_n&\mapsto-\beta x_0^n\otimes1.
\end{align*} 
Then the following statements hold.
\begin{prop}
The map $\phi_{\alpha\beta\gamma g}$ is a surjective homomorphism of associative algebras.
\end{prop}
\begin{proof}
Similar to Proposition \ref{alghomo}, we have
\[
[\phi_{\alpha\beta\gamma g}(L_n),\phi_{\alpha\beta\gamma g}(x_m)]=m\phi_{\alpha\beta\gamma g}(x_{n+m}),
\]
where $x\in\left\{a,b,c,d\right\}$.
By direct computations, for $f(t)\in\C[t]$, the following identities hold in the algebra $\mathcal{D}$:
\begin{align*}
[t\partial_{t},f(t)]&=tf'(t),\\
[\partial_{t},f(t)]&=f'(t),\\
[t\partial_{t},\partial_{t}]&=-\partial_{t}.
\end{align*}
Hence the map $\phi_{\alpha\beta\gamma g}$ is an algebra homomorphism.

We have the following equations:
\begin{align*}
	x_0^{\pm1}\otimes1&=-\beta^{-1}\phi_{\alpha\beta\gamma g}(c_{\pm1}),\\
	\partial_0\otimes1&=\phi_{\alpha\beta\gamma g}(L_0),\\
	1\otimes t&=\phi_{\alpha\beta\gamma g}(a_0),\\
	1\otimes\partial_{t}&=\beta^{-1}\phi_{\alpha\beta\gamma g}\left(b_0-g(a_0)\right).
\end{align*}
The tensor product algebra ${R}_0\otimes\mathcal{D}$ is generated by $\left\{x_0^{\pm1}\otimes1,\partial_0\otimes1,1\otimes t,1\otimes\partial_{t}\right\}$.
It follows that $\phi_{\alpha\beta\gamma g}$ is surjective.
\end{proof}
\begin{cor}
	Every simple ${R}_0\otimes\mathcal{D}$-module can be lifted to a simple $U(L)$-module.
\end{cor}
Next, we focus on a particular simple module.
Let us formulate it here.
For $\lambda\in\C^*$, the polynomial algebra $\C[s]_{\lambda}$ is a simple $\mathcal{R}_0$-module, where for $f(s)\in\C[s]_{\lambda}$, the $\mathcal{R}_0$-actions are as follows:
\[
x_0f(s)=\lambda f(s-1),\quad\quad \partial_0f(s)=sf(s).
\] 
The polynomial algebra $\C[t]$ is naturally a simple $\mathcal{D}$-module, where $t\in\mathcal{D}$ acts as multiplication while $\partial_{t}\in\mathcal{D}$ acts as a derivation.

Now, the tensor product $\C[s]_\lambda\otimes\C[t]$ admits the structure of a $U(L)$-module.
Denote by $\Omega(\alpha,\beta,\gamma,\lambda,g)=\C[s,t]\cong \C[s]_\lambda\otimes\C[t]$.
Then $\Omega(\alpha,\beta,\gamma,\lambda,g)$ is a $U(L)$-module, where for $n\in\Z$ and $f(s,t)\in\Omega(\alpha,\beta,\gamma,\lambda,g)$, the $U(L)$-actions are given as follows:
\begin{equation}\label{ac}
	\begin{split}
		 L_n f(s,t)&=\lambda^n(s+n\alpha)f(s-n,t),\\
		d_n f(s,t)&=\lambda^n\beta^{-1}\left(tg(t)+\gamma\right)f(s-n,t)+\lambda^n t\partial_t f(s-n,t),\\
		a_n f(s,t)&=\lambda^n t f(s-n,t),\\
		b_n f(s,t)&=\lambda^n g(t) f(s-n,t)+\lambda^n \beta \partial_t f(s-n,t),\\
		c_n f(s,t)&=-\lambda^n\beta f(s-n,t).
	\end{split}
\end{equation}  
\begin{prop}
	The following statements are valid.
	\begin{enumerate}[(1)]
		\item 
		$\Omega(\alpha,\beta,\gamma,\lambda,g)$ is a simple $U(L)$-module and is a free $U(B)$-module of rank 1.
		\item
		Suppose $\deg(g)=n\in\N$. Then $\Omega(\alpha,\beta,\gamma,\lambda,g)$ is a free $U(H)$-module of rank $\deg(g)+1$.  
	\end{enumerate}
\end{prop} 
\begin{proof}
\begin{enumerate}[(1)]
	\item 
	The element $1\in \Omega(\alpha,\beta,\gamma,\lambda,g)$ is a $U(B)$-basis, which implies $\Omega(\alpha,\beta,\gamma,\lambda,g)$ is a free $U(B)$-module of rank 1.
	For the simplicity, suppose $N$ is a submodule of $\Omega(\alpha,\beta,\gamma,\lambda,g)$. Then, by the action of $c_1\in L$, we can find some $0\ne f(t)\in\C[t]\cap N$. 
	Combine the actions of $a_0$ and $b_0$, we have $\partial_tf(t)=\beta^{-1}\left(b_0-g(a_0)\right)f(t)$. Hence we have $1\in N$ and $N=\Omega(\alpha,\beta,\gamma,\lambda,g)$.
	\item 
	Let $\mathbb{B}=\left\{1,t,\ldots,t^n\right\}\subset\Omega(\alpha,\beta,\gamma,\lambda,g)$.
	Suppose $g(t)=\sum_{k=1}^{n}g_kt^k$, where $g_k\in\C$ and $g_n\ne0$.
	Then, for $s\in\N$, we have
	\begin{equation}
		\label{gene}
		t^{n+1+s}=(\beta d_0-s-\gamma)t^s-\sum_{k=0}^{n-1}g_kt^{k+1+s}.
	\end{equation}
	We prove by induction that $t^{n+1+s}\in\C[d_0]\mathbb{B}$ for all $s\in\N$. 
	The case $s=0$ follows from Equation \eqref{gene}.
	For $1\le s\le n$, Equation \eqref{gene} and the induction hypothesis show that
	\[
	t^{n+1+s}=\left((\beta d_0-s-\gamma)t^s-\sum_{k=0}^{n-s-1}g_kt^{k+1+s}\right)-\sum_{k=n-s}^{n-1}g_kt^{k+1+s}\in\C[d_0]\mathbb{B}.
	\] 
	Similarly, for $s\ge n+1$, we also have $t^{n+1+s}\in\C[d_0]\mathbb{B}$.
	Now, we conclude that $\Omega(\alpha,\beta,\gamma,\lambda,g)=U(H)\mathbb{B}$.
	
	Next, we prove $\mathbb{B}$ is $U(H)$-linearly independent.
	
	Assume first that there exist $h_0(d_0),\ldots,h_n(d_0)\in\C[d_0]$, which are not all equal to zero, such that
	\[
	\sum_{k=0}^{n}h_k(d_0)t^k=0.
	\]
	For $m\in\N$ and $0\ne f(t)\in\C[t]$, we have
	\begin{equation}
		\label{deg}
		\deg\left(d_0^mf(t)\right)=\deg\left(f(t)\right)+m(n+1).
	\end{equation}
	It follows that there exist $i\ne j$ such that $\deg\left(h_i(d_0)t^i\right)=\deg\left(h_j(d_0)t^j\right)$.
	Hence, from Equation \eqref{deg}, we have $\left|i-j\right|=(n+1)\left|\deg\left(h_i(d_0)\right)-\deg\left(h_j(d_0)\right)\right|$, which is a contradiction since $0<\left|i-j\right|\le n$.
	
	In general, for $0\le k\le n$, suppose
	$
	f_k=\sum_{i=0}^{m}L_0^ih_{ki}(d_0)\in U(H)
	$
	such that $\sum_{k=0}^{m}f_kt^k=0$, where $h_{ki}(d_0)\in\C[d_0]$.
	Then we have
	\[
	\sum_{i=0}^{m}s^i\left(\sum_{k=0}^{n}h_{ki}(d_0)t^k\right)=0.
	\]
	Note that $\sum_{k=0}^{n}h_{ki}(d_0)t^k\in\C[t]$, it follows that $f_k=0$ for all $0\le k\le n$.\qedhere
\end{enumerate}
\end{proof}
Furthermore, the following theorem holds, which is the \textbf{main result} of this section.
\begin{thm}
	Suppose $M$ is a simple $U(L)$-module and is a free $U(B)$-module of rank 1. Then $M$ is isomorphic to $\Omega(\alpha,\beta,\gamma,\lambda,g)$.
\end{thm}
\begin{proof}
	Suppose $v\in M$ is a $U(B)$-basis. Then, from Theorem 3.2 in \cite{HC}, we know that 
	$$L_n v=\lambda^n\left(L_0+np(a_0)\right)v\quad\mbox{and}\quad a_n v=(\lambda^na_0)v$$ 
	for some $\lambda\in\C^*$ and $p(a_0)\in\C[a_0]$. 
	
	Assume that $b_nv=B_n(L_0,a_0)v$, $c_n v=C_n(L_0,a_0)v$ and $d_nv=D_n(L_0,a_0)v$. According to the definition of the algebra $L$ and by induction, for $n\in\Z$, $k\in\N$ and $x_n\in \left \{ L_n,a_n,b_n,c_n,d_n \right \}$, one can verify that the following identities hold in $U(L)$:
	\begin{align*}
		x_n L_0^k=(L_0-n)^k x_n, \quad\quad
		b_n a_0^k=a_0^k b_n - k a_0^{k-1}c_n,\quad\quad
		d_n a_0^k=a_0^k d_n + k a_0^{k-1}a_n.
	\end{align*}   
	It follows that for any $f(L_0,a_0)\in U(B)$, the following equations are valid.
	\begin{align*}
		b_nf(L_0,a_0)v&=B_n(L_0,a_0)f(L_0-n,a_0)v-C_n(L_0,a_0)\partial_{a_0}f(L_0-n,a_0)v,\\
		c_nf(L_0,a_0)v&=C_n(L_0,a_0)f(L_0-n,a_0)v,\\
		d_nf(L_0,a_0)v&=D_n(L_0,a_0)f(L_0-n,a_0)v+\lambda^na_0\partial_{a_0}f(L_0-n,a_0)v.
	\end{align*}
	\begin{claim}
		$C_n(L_0,a_0)\in\C$ for all $n\in\Z$.
		
		In fact, by $[a_1,c_n]=0$, we have $C_n(L_0-1,a_0)=C_n(L_0,a_0)$ for all $n\in\Z$. Hence $C_n(L_0,a_0)=C_n(a_0)\in\C[a_0]$.
		From the equation $[L_m,c_n]=nc_{m+n}$, we deduce that $C_n(a_0)=\lambda^n C_0(a_0)$. 
		The equation $[b_m,c_n]=0$ implies
		\begin{equation}\label{1}
			B_m(L_0-n,a_0)C_n(a_0)=B_m(L_0,a_0)C_n(a_0)-C_m(a_0)\partial_{a_0} C_n(a_0).
		\end{equation} 
		Taking $m=n=0$ in Equation \eqref{1}, we have $C_0(a_0)\partial_{a_0} C_0(a_0)=0$.
		It follows that $C_0(a_0)\in\C$ and thus $C_n(a_0)=\lambda^n C_0(a_0)\in\C$.
		The claim is correct.
	\end{claim}
	Below, we use $C_n$ to denote $C_n(L_0,a_0)\in\C$. Two different cases need to be considered.
	\begin{case}
		$C_0=-\beta\in\C^*$.
		
		In this case, from Equation \eqref{1} we know that $B_m(L_0-n,a_0)=B_m(L_0,a_0)$, and thus $B_m(L_0,a_0)=B_m(a_0)\in\C[a_0]$.
		In the same way, we deduce that $D_m(L_0,a_0)=D_m(a_0)\in\C[a_0]$ by the equation $[c_m,d_n]=0$. The equation $[L_m,b_n]=nb_{m+n}$ suggests
		\begin{equation}\label{2}
			nB_{m+n}(a_0)=n\lambda^m B_n(a_0)+m\lambda^n \partial_{a_0} p(a_0).
		\end{equation} 
		Taking $m=1$ and $n=0$ in Equation \eqref{2}, we have $p(a_0)=\a\in\C$, and hence $B_m(a_0)=\lambda^m B_0(a_0)$. Similarly, using the equation $[L_m,d_n]=nd_{m+n}$, we obtain $D_m(a_0)=\lambda^m D_0(a_0)$. The equation $[b_0,d_0]=b_0$ indicates
		\begin{align*}
			0=a_0\partial_{a_0}B_0(a_0)+B_0(a_0)-\beta\partial_{a_0}D_0(a_0)
			=\partial_{a_0}\left(a_0B_0(a_0)-\beta D_0(a_0)\right).
		\end{align*}
		It follows that $a_0B_0(a_0)-\beta D_0(a_0)=-\gamma\in\C$. 
		We can now see that $M$ is isomorphic to $\Omega(\alpha,\beta,\gamma,\lambda,g)$, where $g(t)=B_0(t)$. 
	\end{case}
	\begin{case}
		$C_0=0$.
		
		In this case, $C_n=\lambda^n C_0=0$. By $[b_m,d_n]=b_{m+n}$, we deduce that
		\begin{equation}\label{3}
			\begin{split}
			B_{m+n}(L_0,a_0)=B&_m(L_0,a_0)D_n(L_0-m,a_0)-B_m(L_0-n,a_0)D_n(L_0,a_0)\\
			&-\lambda^n a_0\partial_{a_0}B_m(L_0-n,a_0).
			\end{split}
		\end{equation}
		Taking $m=n=0$ in Equation \eqref{3}, we have $B_0(L_0,a_0)+a_0\partial{a_0}B_0(L_0,a_0)=0$. 
		It follows that $B_0(L_0,a_0)=0$. Taking $m=0$ in Equation \eqref{3}, we see that $B_n(L_0,a_0)=0$ for all $n\in\Z$.
		Now, the subspace spanned by $\left\{L_0^ia_0^nm\mid i\in\N, n\in\Z_{\ge1}\right\}$ is a proper submodule, in contradiction with $M$ being simple.\qedhere
	\end{case}
\end{proof}
\section{Tensor products of $U(B)$-free modules}
\label{44}
In this section, we study the tensor products of finitely many simple $U(L)$-modules constructed in \eqref{ac}.
First, we give necessary and sufficient conditions for them to be simple.
Then, we determine their isomorphism classes.
Before proceeding formally, let us introduce some notations and known results.
\begin{lemma}[Lemma 2 in \cite{TZ}]
	\label{crucial}
	Suppose $\alpha_1,\ldots,\alpha_m\in\C^*$, $s_1,\ldots,s_m\in\Z_{\ge1}$ and $s_1+\cdots+s_m=s$.
	For $n\in\Z$, $1\le t\le m$ and $s_1+\cdots+s_{t-1}+1\le k\le s_1+\cdots+s_t$, define $f_k(n)=n^{k-1-\sum_{j=1}^{t-1}s_j}\alpha_t^n$. Let $\mathfrak{R}=(y_{pq})$ be the $s\times s$ matrix with $y_{pq}=f_q(p-1)$, where $q=1,2,\ldots,s$ and $p=r+1,r+2,\ldots,r+s$ for some $r\ge0$. Then we have
	\begin{equation*}
		\det\mathfrak{R}=\prod_{j=1}^{m}(s_j-1)^{!!}\alpha_j^{{s_j(s_j+2r-1)}/{2}} \prod_{1\le i<j\le m}(\alpha_j-\alpha _i)^{s_is_j},
	\end{equation*}
	where $m^{!!}=m!\times(m-1)!\times\cdots\times 1!$ for $m\in\Z_{\ge 1}$ and $0^{!!}=1$.
\end{lemma}

For $m\in\Z_{\ge 1}$ and $1\le k\le m$, suppose $\alpha_k,\gamma_k\in\C$, $\beta_k,\lambda_k\in\C^*$, and $g_k=g_k(t_k)\in\C[t_k]$.
As a vector space, we have $\Omega(\alpha_k,\beta_k,\gamma_k,\lambda_k,g_k)=\C[s_k,t_k]$.
Define the tensor product
\begin{equation*}
	\textbf{T}(\boldsymbol{\alpha},\boldsymbol{\beta},\boldsymbol{\gamma},\boldsymbol{\lambda},\boldsymbol{g})=\bigotimes_{k=1}^{m}\Omega(\alpha_k,\beta_k,\gamma_k,\lambda_k,g_k).
\end{equation*}

For a non-zero element
\begin{equation}\label{fffff}
	g=\sum_{(\textbf{p},\textbf{q})\in E}g_{(\textbf{p},\textbf{q})}s_1^{p_1}t_1^{q_1}\otimes \cdots \otimes  s_m^{p_m}t_m^{q_m}\in \textbf{T}(\boldsymbol{\alpha},\boldsymbol{\beta},\boldsymbol{\gamma},\boldsymbol{\lambda},\boldsymbol{g}),
\end{equation}
where $(\textbf{p},\textbf{q})=(p_1,\ldots,p_m,q_1,\ldots,q_m)\in\N^{2m}$, $ g_{(\textbf{p},\textbf{q})}\in\C^*$, and $E\subset\N^{2m}$ is a finite set, we use $\textbf{s}^{\textbf{p}}\textbf{t}^{\textbf{q}}$ to denote $s_1^{p_1}t_1^{q_1}\otimes \cdots \otimes  s_m^{p_m}t_m^{q_m}$.
Thus, the element in \eqref{fffff} can be written as 
\begin{equation*}
	g=\sum_{(\textbf{p},\textbf{q})\in E}g_{(\textbf{p},\textbf{q})} \textbf{s}^{\textbf{p}}\textbf{t}^{\textbf{q}}.
\end{equation*}
Define $P_k=\mathrm{max}\left \{ p_k \mid(\textbf{p},\textbf{q}) \in E\right \}$ and $E_k=\left\{(\textbf{p},\textbf{q}) \in E\mid p_k=P_k\right\}$.
For $n\in\Z$, $X\in\left\{L,a\right\}$, and $g\in\textbf{T}(\boldsymbol{\alpha},\boldsymbol{\beta},\boldsymbol{\gamma},\boldsymbol{\lambda},\boldsymbol{g})$, define a subspace
\begin{equation}
	\label{sub}
	N(X,g)=\mathrm{span}_\C\left\{g,X_ng\mid n\in\Z\right\}.
\end{equation}
Recall that for any $n\ge 1$, the \textbf{lexicographical order} on $\N^n$ is defined as
\begin{equation*}
	(a_1,\ldots,a_n)\succ(b_1,\ldots,b_n)\Leftrightarrow \exists\; 1\le k\le n \quad\mathrm{s.t.}\quad a_1=b_1,\ldots,a_{k-1}=b_{k-1}\quad \mathrm{and}\quad a_k>b_k.
\end{equation*}
Define the \textbf{degree} of $g$ to be the maximal element $ (\textbf{p},\textbf{q})\in E\subset\N^{2m}$, and denote it by $\deg(g)$.
For zero element, we assume its degree is infinitesimal.
For $1\le k\le m$, denote by $e_k=(\delta_{1k},\delta_{2k},\ldots,\delta_{mk})\in\N^m$.

\subsection{Simplicity}
The aim of this subsection is to give a necessary and sufficient condition for $\textbf{T}(\boldsymbol{\alpha},\boldsymbol{\beta},\boldsymbol{\gamma},\boldsymbol{\lambda},\boldsymbol{g})$ to be simple.
A lemma is presented below, whose proof is parallel to that of Proposition 3.2 in \cite{CY}.
\begin{lemma}
	Suppose $\lambda_1,\ldots,\lambda_m$ are pairwise distinct non-zero complex numbers. 
	Then for $n\in\Z$, $1\le k\le m$, and $$0\ne g=\sum_{(\textbf{p},\textbf{q})\in E}g_{(\textbf{p},\textbf{q})} \textbf{s}^{\textbf{p}}\textbf{t}^{\textbf{q}}\in\textbf{T}(\boldsymbol{\alpha},\boldsymbol{\beta},\boldsymbol{\gamma},\boldsymbol{\lambda},\boldsymbol{g}),$$ the following statements are valid.
	\begin{align}
		\label{nnn}
		\sum_{(\textbf{p},\textbf{q})\in E}g_{(\textbf{p},\textbf{q})} \textbf{s}^{\textbf{p}+e_k}\textbf{t}^{\textbf{q}}&\in N(L,g),\\
		\label{nn}
		\sum_{(\textbf{p},\textbf{q})\in E} g_{(\textbf{p},\textbf{q})}\textbf{s}^{\textbf{p}}\textbf{t}^{\textbf{q}+e_k}&\in N(a,g),\\
		\label{n}
		\sum_{(\textbf{p},\textbf{q})\in E_k} g_{(\textbf{p},\textbf{q})}\textbf{s}^{\textbf{p}-P_k e_k}\textbf{t}^{\textbf{q}+e_k}&\in N(a,g).
	\end{align}
\end{lemma}
An immediate consequence of \eqref{nnn} and $\eqref{nn}$ is the following lemma.
\begin{lemma}
	\label{ggg}
	Suppose $\lambda_1,\ldots,\lambda_m$ are pairwise distinct non-zero complex numbers. 
	Then the module $\textbf{T}(\boldsymbol{\alpha},\boldsymbol{\beta},\boldsymbol{\gamma},\boldsymbol{\lambda},\boldsymbol{g})$ is generated by $1^{\otimes m}$.
\end{lemma}
Now, we are in a position to prove the \textbf{main theorem} of this subsection.
\begin{thm}
	\label{simple}
	The $U(L)$-module $\textbf{T}(\boldsymbol{\alpha},\boldsymbol{\beta},\boldsymbol{\gamma},\boldsymbol{\lambda},\boldsymbol{g})$ is simple if and only if $\lambda_1,\ldots,\lambda_m$ are pairwise distinct non-zero complex numbers.
\end{thm}
\begin{proof}
	Assume that $\lambda_1,\ldots,\lambda_m$ are pairwise distinct non-zero complex numbers.
	Suppose $N$ is a non-zero submodule of $\textbf{T}(\boldsymbol{\alpha},\boldsymbol{\beta},\boldsymbol{\gamma},\boldsymbol{\lambda},\boldsymbol{g})$, and take $$0\ne g=\sum_{(\textbf{p},\textbf{q})\in E}g_{(\textbf{p},\textbf{q})} \textbf{s}^{\textbf{p}}\textbf{t}^{\textbf{q}}\in N$$ with the minimal degree.
	\begin{claim}
		$\deg(g)=\mathbf{0}.$
	\end{claim}
	Assume that $\deg(g)=(\textbf{p}',\textbf{q}')\ne \mathbf{0}$, where $\textbf{p}'=(p'_1,\ldots,p'_m)$ and $\textbf{q}'=(q'_1,\ldots,q'_m)$. 
	
	If $\textbf{p}'\ne\mathbf{0}$, let $k_0=\min\left \{ k\mid 1\le k\le m,p'_k>0 \right \}$, then from \eqref{n}, we know that
	\begin{equation*}
		0\ne\sum_{(\textbf{p},\textbf{q})\in E_{k_0}}g_{(\textbf{p},\textbf{q})} \textbf{s}^{\textbf{p}-P_{k_0} e_{k_0}}\textbf{t}^{\textbf{q}+e_{k_0}}\in N,   
	\end{equation*}
	whose degree is lower than that of $g$. It follows that $\textbf{p}'=\mathbf{0}$ and $g=\sum_{(\textbf{0},\textbf{q})\in E} g_{(\mathbf{0},\textbf{q})}\textbf{t}^{\textbf{q}}$.
	 
	The action of $b_n$ implies
	\begin{align*}
		b_ng&=\sum_{(\textbf{0},\textbf{q})\in E}\sum_{k=1}^{m}\lambda_k^n g_{(\textbf{0},\textbf{q})}t_1^{q_1}\otimes\cdots\otimes\big(g_k(t_k)t_k^{q_k}+\beta_kq_kt_k^{q_k-1}\big)\otimes\cdots\otimes s_m^{p_m}t_m^{q_m}\in N.
	\end{align*}
	Using the Vandermonde determinant and \eqref{nn}, we know that $\partial_{t_k}g\in N$ for $1\le k\le m$. If $\textbf{q}'\ne\mathbf{0}$, let $k_1=\min\left \{ k\mid 1\le k \le m,q'_k>0 \right \} $, then we see that $0\ne\partial_{t_{k_1}} g\in N$, whose degree is lower than $\deg(g)$. 
	Hence we have $\deg(g)=\mathbf{0}$.
	From Lemma \ref{ggg}, we conclude that $N=\textbf{T}(\boldsymbol{\alpha},\boldsymbol{\beta},\boldsymbol{\gamma},\boldsymbol{\lambda},\boldsymbol{g})$. 
	 
	Conversely, if $\lambda_1,\ldots,\lambda_m$ are not pairwise distinct non-zero complex numbers, without loss of generality, assume that $\lambda_1=\lambda_2$.
	We identify $\C[s_1,t_1]\otimes\C[s_2,t_2]$ with $\C[s_1,s_2,t_1,t_2]$.
	Let $W=\mathrm{span}_\C\left \{ \C[t_1,t_2](s_1+s_2)^p\mid p\in\N \right \} $.
	Direct computations show that $W$ is a proper submodule.
\end{proof}
\subsection{Isomorphism classes}
Subsequently, we determine the isomorphism classes of these simple tensor product modules.
Throughout this subsection, we always assume $\lambda_1,\ldots,\lambda_m\in\C^*$ are pairwise distinct due to Theorem \ref{simple}.

For $g\in\textbf{T}(\boldsymbol{\alpha},\boldsymbol{\beta},\boldsymbol{\gamma},\boldsymbol{\lambda},\boldsymbol{g})$, define
\begin{equation*}
	R_g=\dim\mathrm{span}_\C\left\{ g,a_ng,c_ng\mid n\in\Z\right\} 
\end{equation*}
and
\begin{equation*}
	R_{\textbf{T}}=\inf\left \{ R_g\mid g\ne0 \right \}.
\end{equation*}
Let $\mathbf{N}=\C[t_1]\otimes\cdots\otimes\C[t_m]\subset\textbf{T}(\boldsymbol{\alpha},\boldsymbol{\beta},\boldsymbol{\gamma},\boldsymbol{\lambda},\boldsymbol{g})$.
Then the following proposition holds.
\begin{prop}
	\label{rank}
	For any $0\ne g\in\textbf{T}(\boldsymbol{\alpha},\boldsymbol{\beta},\boldsymbol{\gamma},\boldsymbol{\lambda},\boldsymbol{g})$, we have $R_g\ge m+1$, where the equality holds if and only if $0\ne g\in \mathbf{N}$.
\end{prop}
\begin{proof}
	For a non-zero
	$$g=\sum_{(\textbf{p},\textbf{q})\in E}g_{(\textbf{p},\textbf{q})} \textbf{s}^{\textbf{p}}\textbf{t}^{\textbf{q}}\in\textbf{T}(\boldsymbol{\alpha},\boldsymbol{\beta},\boldsymbol{\gamma},\boldsymbol{\lambda},\boldsymbol{g}),$$
	by the action of $a_n$, we have
	\begin{equation*}
		\begin{split}
			a_n g
			=& \sum_{(\textbf{p},\textbf{q})\in E}\sum_{k = 1}^{m} g_{(\textbf{p},\textbf{q})}s_1^{p_1}t_1^{q_1}\otimes\cdots\otimes \lambda_k^n(s_k-n)^{p_k}t_k^{q_k+1}\otimes \cdots \otimes s_m^{p_m}t_m^{q_m}\\
			=& \sum_{(\textbf{p},\textbf{q})\in E}\sum_{k = 1}^{m}\sum_{x=0}^{p_k}(-1)^xn^x\lambda_k^n \left(g_{(\textbf{p},\textbf{q})} s_1^{p_1}t_1^{q_1}\otimes\cdots\otimes \begin{pmatrix}
				p_k \\
				x
			\end{pmatrix} s_k^{p_k-x}t_k^{q_k+1}\otimes \cdots \otimes s_m^{p_m}t_m^{q_m}\right).
		\end{split}
	\end{equation*}
	Thus, using Lemma \ref{crucial} we deduce that 
	\begin{equation*}
		a_{xk}=\sum_{(\textbf{p},\textbf{q})\in E}g_{(\textbf{p},\textbf{q})}s_1^{p_1}t_1^{q_1}\otimes\cdots\otimes \begin{pmatrix}
			p_k \\
			x
		\end{pmatrix} s_k^{p_k-x}t_k^{q_k+1}\otimes \cdots \otimes s_m^{p_m}t_m^{q_m}\in \mathrm{span}_\C\left \{ g,a_ng,c_ng\mid n\in\Z\right\},
	\end{equation*}
	where $0\le x\le P_k$ and $1\le k\le m$.
	
	The action of $c_n$ implies
	\begin{align*}
		c_ng&=-\sum_{(\textbf{p},\textbf{q})\in E}\sum_{k=1}^{m} g_{(\textbf{p},\textbf{q})}s_1^{p_1}t_1^{q_1}\otimes\cdots\otimes \lambda_k^n\beta_k(s_k-n)^{p_k}t_k^{q_k}\otimes\cdots\otimes s_m^{p_m}t_m^{q_m}\\
		&=-\sum_{(\textbf{p},\textbf{q})\in E}\sum_{k=1}^{m}\sum_{z=0}^{p_k}(-1)^z n^z\lambda_k^n\beta_k \left(g_{(\textbf{p},\textbf{q})}s_1^{p_1}t_1^{q_1}\otimes\cdots\otimes \begin{pmatrix}
			p_k\\
			z
		\end{pmatrix}s_k^{p_k-z}t_k^{q_k}\otimes\cdots\otimes s_m^{p_m}t_m^{q_m}\right).
	\end{align*}
	Lemma \ref{crucial} indicates
	\begin{equation*}
		c_{zk}=\sum_{(\textbf{p},\textbf{q})\in E}g_{(\textbf{p},\textbf{q})}s_1^{p_1}t_1^{q_1}\otimes\cdots\otimes \begin{pmatrix}
			p_k\\
			z
		\end{pmatrix}s_k^{p_k-z}t_k^{q_k}\otimes\cdots\otimes s_m^{p_m}t_m^{q_m}\in\mathrm{span}_\C\left\{ g,a_ng,c_ng\mid n\in\Z\right\},
	\end{equation*}
	where $0\le z\le P_k$ and $1\le k\le m$. 
	It follows that
		$$\mathrm{span}_{\C}\left \{ g,a_ng,c_ng\mid n\in\Z\right \}=\mathrm{span}_\C\left \{ g,a_{zk},c_{zk}\mid 1\le k\le m,0\le z\le P_k\right\}.$$
		
	Assume that $\deg(g)=(\textbf{p}',\textbf{q}')$. Then, for $1\le k\le m$, we have $\deg(a_{0k})=(\textbf{p}',\textbf{q}'+e_k)$.
	\setcounter{case}{0}
	\begin{case}
		$\textbf{p}'=\mathbf{0}$.
		
		In this case, we have $\textbf{p}=\mathbf{0}$ for all $(\textbf{p},\textbf{q})\in E$, and thus $P_k=0$ for $1\le k\le m$. In combination with the fact that $c_{0k}=g$, we infer that $$R_g=\dim\mathrm{span}_\C\left \{ g,a_{0k},c_{0k}\mid 1\le k\le m\right \}=m+1.$$ 
	\end{case}
	\begin{case}
		$\textbf{p}'\ne\mathbf{0}$.
		
		Let $k_2=\min\left \{ k\mid 1\le k\le m,p'_k>0 \right \} $. Then we have $\deg(c_{P_{k_2}k_2})\prec\deg(g)$ and hence
		$$\dim\mathrm{span}_{\C}\left\{ g,a_{0k},c_{P_{k_2}k_2}\mid 1\le k\le m\right\}=m+2,$$
		which implies $R_g>m+1$.\qedhere
	\end{case}
\end{proof}

Suppose
$$\textbf{T}(\boldsymbol{\xi},\boldsymbol{\epsilon},\boldsymbol{\omega},\boldsymbol{\mu},\boldsymbol{h})=\bigotimes_{k=1}^{r}\Omega(\xi_k,\epsilon_k,\omega_k,\mu_k,h_k)$$
is another simple $U(L)$-module.
Now, we prove the \textbf{main theorem} of this subsection.
\begin{thm}
	The simple tensor product modules
	$$\textbf{T}(\boldsymbol{\alpha},\boldsymbol{\beta},\boldsymbol{\gamma},\boldsymbol{\lambda},\boldsymbol{g})\cong \textbf{T}(\boldsymbol{\xi},\boldsymbol{\epsilon},\boldsymbol{\omega},\boldsymbol{\mu},\boldsymbol{h})$$
	if and only if $m=r$ and $$(\alpha_k,\beta_k,\gamma_k,\lambda_k,g_k)=(\xi_k,\epsilon_k,\omega_k,\mu_k,h_k)\quad\mbox{for all  } 1\le k\le m$$ up to a permutation.
\end{thm}
\begin{proof}
	That $m=r$ follows from Proposition \ref{rank} directly.
	Suppose
	\[
	\phi:\textbf{T}(\boldsymbol{\alpha},\boldsymbol{\beta},\boldsymbol{\gamma},\boldsymbol{\lambda},\boldsymbol{g})\to \textbf{T}(\boldsymbol{\xi},\boldsymbol{\epsilon},\boldsymbol{\omega},\boldsymbol{\mu},\boldsymbol{h})
	\]
	is an isomorphism of $U(L)$-modules.
	Then, from Proposition \ref{rank}, we have
	\begin{equation}
		\label{abc}
		\phi(1^{\otimes m})=\sum_{(\textbf{0},\textbf{q})\in E}g_{\textbf{q}}t_1^{q_1}\otimes\cdots\otimes t_m^{q_m},
	\end{equation}
	where $g_{\textbf{q}}\in\C^*$.
	The action of $L_n$ implies
	\begin{equation}
		\label{syy}
	\begin{split}
	 &\;\;\sum_{k=1}^{m}\lambda_k^n\phi(1\otimes\cdots\otimes s_k\otimes\cdots\otimes1)+\sum_{k=1}^{m}n\lambda_k^n\alpha_k\phi(1^{\otimes m})\\
	 =&\sum_{(\textbf{0},\textbf{q})\in E}\sum_{k=1}^m\mu_k^ng_{\textbf{q}} \left(t_1^{q_1}\otimes\cdots\otimes s_kt_k^{q_k}\otimes\cdots\otimes t_m^{q_m}\right)
	 +\sum_{(\textbf{0},\textbf{q})\in E}\sum_{k=1}^m n\mu_k^n\xi_k\phi(1^{\otimes m})
	\end{split}
	\end{equation}
	Since $\phi(1\otimes\cdots\otimes s_k\otimes\cdots\otimes1)\ne0$, using Lemma \ref{crucial}, we conclude that $$\left\{\lambda_1,\ldots,\lambda_m\right\}=\left\{\mu_1,\ldots,\mu_m\right\}.$$
	Without loss of generality, assume that $\lambda_k=\mu_k$ for $1\le k\le m$.
	Then, Lemma \ref{crucial} and Equation \eqref{syy} imply $\alpha_k=\xi_k$.
	By Equation \eqref{abc}, the action of $a_n$, and the Vandermonde determinant, the equation
	\[
	\phi(t_1^{r_1}\otimes\cdots\otimes t_m^{r_m})=\sum_{(\textbf{0},\textbf{q})\in E}g_{\textbf{q}}t_1^{r_1+q_1}\otimes\cdots\otimes t_m^{r_m+q_m}.
	\]
	holds for all $r_k\in\N$.
	The action of $b_n$ together with the above equation and the Vandermonde determinant show that the equation
	\[
	\sum_{(\textbf{0},\textbf{q})\in E}g_{\textbf{q}}t_1^{q_1}\otimes\cdots\otimes\left(g_k(t_k)-h_k(t_k)\right)t_k^{q_k}\otimes\cdots\otimes t_m^{q_m}=\sum_{(\textbf{0},\textbf{q})\in E}g_{\textbf{q}}t_1^{q_1}\otimes\epsilon_kq_kt_k^{q_k-1}\otimes\cdots\otimes t_m^{q_m}
	\]
	holds for $1\le k\le m$.
	It follows that $g_k(t_k)=h_k(t_k)$ and $\phi(1^{\otimes m})\in\C^*(1^{\otimes m})$.
	We can now see that $\beta_k=\epsilon_k$ and $\gamma_k=\omega_k$.
	The proof is completed.
\end{proof}
\section*{Acknowledgements}
The author would like to acknowledge the support from National Key R\&D Program of China (2024YFA1013802), NSF of China (11931009, 12101152, 12161141001, 12171132, 12401030, and 12401036), Innovation Program for Quantum Science and Technology (2021ZD0302902), and the Fundamental Research Funds for the Central Universities of China.


\begin{thebibliography}{99}

\bibitem{B}
A. Babichenko et al., Representations of the Nappi-Witten vertex operator algebra, Lett. Math. Phys. {\bf 111} (2021), no.~5, Paper No. 131, 30 pp.; MR4323158

\bibitem{BJP}
Y. Bao, C.~P. Jiang and Y.~F. Pei, Representations of affine Nappi-Witten algebras, J. Algebra {\bf 342} (2011), 111--133; MR2824531

\bibitem{CLW}
Y. Cai, R. L\"u{} and Y. Wang, Classification of simple Harish-Chandra modules for map (super)algebras related to the Virasoro algebra, J. Algebra {\bf 570} (2021), 397--415; MR4188307

\bibitem{CX}
H. Chen and D. Xu, Modules over the affine-Virasoro algebra of Nappi-Witten type, J. Algebra {\bf 657} (2024), 549--580; MR4756903

\bibitem{CH}
Q. Chen and J.~Z. Han, Non-weight modules over the affine-Virasoro algebra of type $A_1$, J. Math. Phys. {\bf 60} (2019), no.~7, 071707, 9 pp.; MR3985468

\bibitem{CY}
Q. Chen and Y.~F. Yao, A new class of irreducible modules over the affine-Virasoro algebra of type $A_1$, J. Algebra {\bf 608} (2022), 553--572; MR4447741

\bibitem{D}
B.~F. Dubsky et al., Simple modules over the Lie algebras of divergence zero vector fields on a torus, Forum Math. {\bf 31} (2019), no.~3, 727--741; MR3943336

\bibitem{GHL}
X. Guo, X. Huo and X.~W. Liu, New simple $\widehat{\mathfrak{sl}}_2$-modules from Weyl algebra modules, J. Algebra {\bf 634} (2023), 44--73; MR4620669

\bibitem{HC}
J.~Z. Han, Q. Chen and Y. Su, Modules over the algebra $\mathcal{V}ir(a,b)$, Linear Algebra Appl. {\bf 515} (2017), 11--23; MR3588533

\bibitem{K}
V.~G. Kac, Highest weight representations of conformal current algebras, in {\it Topological and geometrical methods in field theory (Espoo, 1986)}, 3--15, World Sci. Publ., Teaneck, NJ, ; MR1026476

\bibitem{L}
T.~A. Larsson, Conformal fields: a class of representations of ${\rm Vect}(N)$, Internat. J. Modern Phys. A {\bf 7} (1992), no.~26, 6493--6508; MR1186508

\bibitem{LPX}
D. Liu, Y.~F. Pei and L.~M. Xia, Classification of quasi-finite irreducible modules over affine Virasoro algebras, J. Lie Theory {\bf 31} (2021), no.~2, 575--582; MR4225044

\bibitem{M}
W. Miller Jr., {\it Lie theory and special functions}, Mathematics in Science and Engineering, Vol. 43, Academic Press, New York-London, 1968; MR0264140

\bibitem{NW}
C.~R. Nappi and E. Witten, Wess-Zumino-Witten model based on a nonsemisimple group, Phys. Rev. Lett. {\bf 71} (1993), no.~23, 3751--3753; MR1247150

\bibitem{N}
J. Nilsson, Simple $\mathfrak{sl}(V)$-modules which are free over an abelian subalgebra, Forum Math. {\bf 35} (2023), no.~5, 1237--1255; MR4635353

\bibitem{S}
G.~Y. Shen, Graded modules of graded Lie algebras of Cartan type. I. Mixed products of modules, Sci. Sinica Ser. A {\bf 29} (1986), no.~6, 570--581; MR0862418

\bibitem{TZ}
H. Tan and K. Zhao, Irreducible Virasoro modules from tensor products (II), J. Algebra {\bf 394} (2013), 357--373; MR3092725

\bibitem{UKO}
S. U\u guz, I.~A. Karimjanov and B.~A. Omirov, Leibniz algebras associated with representations of the Diamond Lie algebra, Algebr. Represent. Theory {\bf 20} (2017), no.~1, 175--195; MR3606486

\end{thebibliography}
\end{document}